\documentclass[12pt,twoside]{amsart}
\usepackage{amsmath, amsthm, amscd, amsfonts, amssymb, graphicx}
\usepackage[bookmarksnumbered, plainpages]{hyperref}
\usepackage[latin1]{inputenc}

\newtheorem{thm}{Theorem}[section]

\newtheorem{lem}[thm]{Lemma}

\newtheorem{defn}[thm]{Definition}

\numberwithin{equation}{section}

\begin{document}

\title{\bf Canonical connections and algebraic Wanas solitons of three-dimensional Lorentzian Lie groups}
\author{Yong Wang$^*$, Yue Tang}

\thanks{{\scriptsize
\hskip -0.4 true cm \textit{2010 Mathematics Subject Classification:}
53C40; 53C42.
 }}

\maketitle

\begin{abstract}
 In this paper, we compute the Wanas tensor associated to canonical connections on three-dimensional Lorentzian Lie groups with
 some product structure. We define algebraic Wanas solitons associated to canonical connections. We classify algebraic Wanas solitons associated to canonical connections on three-dimensional Lorentzian Lie groups with
 some product structure.
\end{abstract}

\vskip 0.2 true cm


\pagestyle{myheadings}
\markboth{\rightline {\scriptsize Wang}}
         {\leftline{\scriptsize Canonical connections and algebraic Wanas solitons}}

\bigskip
\bigskip


\section{ Introduction}
\indent The concept of the algebraic Ricci soliton was first introduced by Lauret in Riemannian case in \cite{La}. In \cite{La}, Lauret studied the relation
between solvsolitons and Ricci solitons on Riemannian manifolds. More precisely, he proved that any Riemannian solvsoliton metric is a Ricci soliton.
The concept of the algebraic Ricci soliton was extended to the pseudo-Riemannian case in \cite{BO}. In \cite{BO}, Batat and Onda studied
algebraic Ricci solitons of three-dimensional Lorentzian Lie groups. They got a complete classification of algebraic Ricci solitons of three-dimensional Lorentzian Lie groups and they proved that, contrary to the Riemannian case, Lorentzian Ricci solitons needed not be algebraic Ricci solitons.
In \cite{ES}, Etayo and Santamaria studied some affine connections on manifolds with the product structure
or the complex structure. In particular, the canonical connection and the Kobayashi-Nomizu connection for a product structure was studied. In \cite{Wa},
we introduced a particular product structure on three-dimensional Lorentzian Lie groups and we computed canonical connections and Kobayashi-Nomizu connections and their curvature on three-dimensional Lorentzian Lie groups with
  this product structure. We defined algebraic Ricci solitons associated to canonical connections and Kobayashi-Nomizu connections. We classified algebraic Ricci solitons associated to canonical connections and Kobayashi-Nomizu connections on three-dimensional Lorentzian Lie groups with
 this product structure. In \cite{YE}, a Wanas tensor associated to a affine connection on a parallelizable manifold was introduced. In fact,
  a Wanas tensor associated to a affine connection on any manifolds can be defined. In this paper, we compute the Wanas tensor associated to canonical connections on three-dimensional Lorentzian Lie groups with
  the product structure in \cite{Wa}. We define algebraic Wanas solitons associated to canonical connections. We classify algebraic Wanas solitons associated to canonical connections on three-dimensional Lorentzian Lie groups with the product structure in \cite{Wa}.\\
\indent In Section 2, We classify algebraic Wanas solitons associated to canonical connections on three-dimensional unimodular Lorentzian Lie groups with
 the product structure.
In Section 3, we classify algebraic Wanas solitons associated to canonical connections on three-dimensional non-unimodular Lorentzian Lie groups with
 the product structure.


\vskip 1 true cm

\section{ Algebraic Wanas solitons associated to canonical connections on three-dimensional unimodular Lorentzian Lie groups}

Three-dimensional Lorentzian Lie groups had been classified in \cite{Ca1,CP}(see Theorem 2.1 and Theorem 2.2 in \cite{BO}). Throughout this paper,
we shall by $\{G_i\}_{i=1,\cdots,7}$, denote the connected, simply connected three-dimensional Lie group equipped with a left-invariant Lorentzian metric $g$ and
having Lie algebra $\{\mathfrak{g}\}_{i=1,\cdots,7}$. Let $\nabla$ be the Levi-Civita connection of $G_i$ and $R$ its curvature tensor, taken with the convention
\begin{equation}
R(X,Y)Z=\nabla_X\nabla_YZ-\nabla_Y\nabla_XZ-\nabla_{[X,Y]}Z.
\end{equation}
We define a product structure $J$ on $G_i$ as in \cite{Wa} by
\begin{equation}Je_1=e_1,~Je_2=e_2,~Je_3=-e_3,
\end{equation}
then $J^2={\rm id}$ and $g(Je_j,Je_j)=g(e_j,e_j)$. By \cite{ES}, we define the canonical connection as follows:
\begin{equation}\nabla^0_XY=\nabla_XY-\frac{1}{2}(\nabla_XJ)JY,
\end{equation}
We define its curvature tensor by
\begin{equation}
R^0(X,Y)Z=\nabla^0_X\nabla^0_YZ-\nabla^0_Y\nabla^0_XZ-\nabla^0_{[X,Y]}Z.
\end{equation}
The Ricci tensors of $(G_i,g)$ associated to the canonical connection
is defined by
\begin{equation}\rho^0(X,Y)=-g(R^0(X,e_1)Y,e_1)-g(R^0(X,e_2)Y,e_2)+g(R^0(X,e_3)Y,e_3),
\end{equation}
where $\{e_1,e_2,e_3\}$ is a pseudo-orthonormal basis, with $e_3$ timelike.
 The Ricci operators ${\rm Ric}^0$ is given by
\begin{equation}\rho^0(X,Y)=g({\rm Ric}^0(X),Y).
\end{equation}
Let
\begin{equation}\widetilde{\rho}^0(X,Y)=\frac{{\rho}^0(X,Y)+{\rho}^0(Y,X)}{2}
,~~\widetilde{\rho}^0(X,Y)=g(\widetilde{{\rm Ric}}^0(X),Y).
\end{equation}
Let the torsion tensor $T^0$ associated to $\nabla^0$ be
\begin{equation}T^0(X,Y)=\nabla^0_XY-\nabla^0_YX-[X,Y].
\end{equation}
We define the Wanas tensor $W^0$ associated to $\nabla^0$ by (see Theorem 4.1 in \cite{YE}):
\begin{equation}W^0(X,Y)(Z):=R^0(X,Y)Z-A^0(X,Y)Z,~~A^0(X,Y)Z:=T^0(T^0(X,Y),Z).
\end{equation}
We define
\begin{equation}w^0(X,Y)=-g(W^0(X,e_1)Y,e_1)-g(W^0(X,e_2)Y,e_2)+g(W^0(X,e_3)Y,e_3),
\end{equation}
and
\begin{equation}a^0(X,Y)=-g(A^0(X,e_1)Y,e_1)-g(A^0(X,e_2)Y,e_2)+g(A^0(X,e_3)Y,e_3).
\end{equation}
Then $w^0(X,Y)=\rho^0(X,Y)-a^0(X,Y)$. We define
\begin{equation}a^0(X,Y)=g(\overline{A}^0(X),Y),~~w^0(X,Y)=g({\rm Wan}^0(X),Y),
\end{equation}
\begin{equation}
\widetilde{w}^0(X,Y)=\frac{w^0(X,Y)+w^0(Y,X)}{2},
~~\widetilde{w}^0(X,Y)=g(\widetilde{{\rm Wan}}^0(X),Y).
\end{equation}
Then
\begin{equation}{\rm Wan}^0(X)={\rm Ric}^0(X)-\overline{A}^0(X).
\end{equation}
\vskip 0.5 true cm
\begin{defn}
$(G_i,g,J)$ is called the first (resp. second) kind algebraic Wanas soliton associated to the connection $\nabla^0$ if it satisfies
\begin{equation}
{\rm Wan}^0=c{\rm Id}+D~~({\rm resp.}~\widetilde{{\rm Wan}}^0=c{\rm Id}+D),
\end{equation}
where $c$ is a real number, and $D$ is a derivation of $\mathfrak{g}$, that is
\begin{equation}D[X,Y]=[DX,Y]+[X,DY]~~{\rm for }~ X,Y\in \mathfrak{g}.
\end{equation}
\end{defn}
\vskip 0.5 true cm
\noindent{\bf 2.1 Algebraic Wanas solitons of $G_1$}\\
\vskip 0.5 true cm
By (2.1) in \cite{BO}, we have for $G_1$, there exists a pseudo-orthonormal basis $\{e_1,e_2,e_3\}$ with $e_3$ timelike such that the Lie
algebra of $G_1$ satisfies
\begin{equation}
[e_1,e_2]=\alpha e_1-\beta e_3,~~[e_1,e_3]=-\alpha e_1-\beta e_2,~~[e_2,e_3]=\beta e_1+\alpha e_2+\alpha e_3,~~\alpha\neq 0.
\end{equation}
By Lemma 2.4 in \cite{Wa}, we have
\begin{lem}
The canonical connection $\nabla^0$ of $(G_1,J)$ is given by
\begin{align}
&\nabla^0_{e_1}e_1=-\alpha e_2,~~\nabla^0_{e_1}e_2=\alpha e_1,~~\nabla^0_{e_1}e_3=0,\\\notag
&\nabla^0_{e_2}e_1=0,~~\nabla^0_{e_2}e_2=0,~~\nabla^0_{e_2}e_3=0,\\\notag
&\nabla^0_{e_3}e_1=\frac{\beta}{2}e_2,~~\nabla^0_{e_3}e_2=-\frac{\beta}{2}e_1,~~\nabla^0_{e_3}e_3=0.
\notag
\end{align}
\end{lem}
By (2.8) and Lemma 2.2, we have
\begin{align}
&T^0(e_1,e_2)=\beta e_3,~~T^0(e_1,e_3)=\alpha e_1+\frac{\beta}{2} e_2,~~T^0(e_2,e_3)=-\frac{\beta}{2} e_1-\alpha e_2-\alpha e_3.
\end{align}
By (2.9) and (2.19), we have
\vskip 0.5 true cm
\begin{lem}
The tensor $A^0$ of the canonical connection $\nabla^0$ of $(G_1,J)$ is given by
\begin{align}
&A^0(e_1,e_2)e_1=-\alpha\beta e_1-\frac{\beta^2}{2} e_2,~~A^0(e_1,e_2)e_2=\frac{\beta^2}{2}e_1+\alpha\beta e_2+\alpha\beta e_3,~~A^0(e_1,e_2)e_3=0,\\\notag
&A^0(e_1,e_3)e_1=-\frac{\beta^2}{2}e_3,~~A^0(e_1,e_3)e_2=\alpha\beta e_3,~~A^0(e_1,e_3)e_3=(\alpha^2-\frac{\beta^2}{4})e_1-\frac{\alpha\beta}{2}e_3,\\\notag
&A^0(e_2,e_3)e_1=\alpha^2 e_1+\frac{\alpha\beta}{2}e_2+\alpha\beta e_3,
~~A^0(e_2,e_3)e_2=-\frac{\alpha\beta}{2}e_1-\alpha^2 e_2-(\alpha^2+\frac{\beta^2}{2})e_3,\\\notag
&
A^0(e_2,e_3)e_3=
(\alpha^2-\frac{\beta^2}{4})e_2+\alpha^2e_3.\notag
\end{align}
\end{lem}
\vskip 0.5 true cm
By Lemma 2.3, (2.11) and (2.12), we have
\begin{align}
{\overline{A}}^0\left(\begin{array}{c}
e_1\\
e_2\\
e_3
\end{array}\right)=\left(\begin{array}{ccc}
\beta^2&-2\alpha\beta&-\frac{\alpha\beta}{2}\\
-2\alpha\beta&\alpha^2+\beta^2&\alpha^2\\
\frac{\alpha\beta}{2}&-\alpha^2&-2\alpha^2+\frac{\beta^2}{2}
\end{array}\right)\left(\begin{array}{c}
e_1\\
e_2\\
e_3
\end{array}\right).
\end{align}
By (2.22) in \cite{Wa}, we have
\begin{align}
{\rm Ric}^0\left(\begin{array}{c}
e_1\\
e_2\\
e_3
\end{array}\right)=\left(\begin{array}{ccc}
-\left(\alpha^2+\frac{\beta^2}{2}\right)&0&0\\
0&-\left(\alpha^2+\frac{\beta^2}{2}\right)&0\\
\frac{\alpha\beta}{2}&\alpha^2&0
\end{array}\right)\left(\begin{array}{c}
e_1\\
e_2\\
e_3
\end{array}\right).
\end{align}
By (2.14), we have
\begin{align}
{{\rm Wan}}^0\left(\begin{array}{c}
e_1\\
e_2\\
e_3
\end{array}\right)=\left(\begin{array}{ccc}
-(\alpha^2+\frac{3}{2}\beta^2)
&2\alpha\beta&\frac{\alpha\beta}{2}\\
2\alpha\beta&-(2\alpha^2+\frac{3}{2}\beta^2)
&-\alpha^2\\
0&2\alpha^2&2\alpha^2-\frac{\beta^2}{2}
\end{array}\right)\left(\begin{array}{c}
e_1\\
e_2\\
e_3
\end{array}\right).
\end{align}
\noindent If $(G_1,g,J)$ is the first kind algebraic Wanas soliton associated to the connection $\nabla^0$, then
${\rm Wan}^0=c{\rm Id}+D$, so
\begin{align}
\left\{\begin{array}{l}
De_1=-(\alpha^2+\frac{3}{2}\beta^2+c)e_1
+2\alpha\beta e_2+\frac{\alpha\beta}{2}e_3,\\
De_2=2\alpha\beta e_1-(2\alpha^2+\frac{3}{2}\beta^2+c)e_2
-\alpha^2e_3\\
De_3=2\alpha^2 e_2+(2\alpha^2-\frac{\beta^2}{2}-c) e_3.\\
\end{array}\right.
\end{align}
By (2.16) and (2.24), we can get $\alpha=0$ which is a contradiction with (2.17). So we have
\vskip 0.5 true cm
\begin{thm}
$(G_1,g,J)$ is not the first kind algebraic Wanas soliton associated to the connection $\nabla^0$.
\end{thm}
\vskip 0.5 true cm
By (2.12), (2.13) and (2.23), we get
\begin{align}
{\widetilde{{\rm Wan}}}^0\left(\begin{array}{c}
e_1\\
e_2\\
e_3
\end{array}\right)=\left(\begin{array}{ccc}
-(\alpha^2+\frac{3}{2}\beta^2)
&2\alpha\beta&\frac{\alpha\beta}{4}\\
2\alpha\beta&-(2\alpha^2+\frac{3}{2}\beta^2)
&-\frac{3}{2}\alpha^2\\
-\frac{\alpha\beta}{4}&\frac{3}{2}\alpha^2&2\alpha^2-\frac{\beta^2}{2}
\end{array}\right)\left(\begin{array}{c}
e_1\\
e_2\\
e_3
\end{array}\right).
\end{align}
\noindent If $(G_1,g,J)$ is the second kind algebraic Wanas soliton associated to the connection $\nabla^0$, then
$\widetilde{{\rm Wan}}^0=c{\rm Id}+D$, so
\begin{align}
\left\{\begin{array}{l}
De_1=-(\alpha^2+\frac{3}{2}\beta^2+c)e_1
+2\alpha\beta e_2+\frac{\alpha\beta}{4}e_3,\\
De_2=2\alpha\beta e_1-(2\alpha^2+\frac{3}{2}\beta^2+c)e_2
-\frac{3}{2}\alpha^2e_3\\
De_3=-\frac{\alpha\beta}{4}e_1+\frac{3}{2}\alpha^2 e_2+(2\alpha^2-\frac{\beta^2}{2}-c) e_3.\\
\end{array}\right.
\end{align}
By (2.16) and (2.26), similar to Theorem 2.4, Then
\vskip 0.5 true cm
\begin{thm}
$(G_1,g,J)$ is not the second kind algebraic Wanas soliton associated to the connection $\nabla^0$.
\end{thm}

\noindent{\bf 2.2 Algebraic Wanas solitons of $G_2$}\\

By (2.2) in \cite{BO}, we have for $G_2$, there exists a pseudo-orthonormal basis $\{e_1,e_2,e_3\}$ with $e_3$ timelike such that the Lie
algebra of $G_2$ satisfies
\begin{equation}
[e_1,e_2]=\gamma e_2-\beta e_3,~~[e_1,e_3]=-\beta e_2-\gamma e_3,~~[e_2,e_3]=\alpha e_1,~~\gamma\neq 0.
\end{equation}
\noindent By Lemma 2.14 in \cite{Wa}, we have
\vskip 0.5 true cm
\begin{lem}
The canonical connection $\nabla^0$ of $(G_2,J)$ is given by
\begin{align}
&\nabla^0_{e_1}e_1=0,~~\nabla^0_{e_1}e_2=0,~~\nabla^0_{e_1}e_3=0,\\\notag
&\nabla^0_{e_2}e_1=-\gamma e_2,~~\nabla^0_{e_2}e_2=\gamma e_1,~~\nabla^0_{e_2}e_3=0,\\\notag
&\nabla^0_{e_3}e_1=\frac{\alpha}{2}e_2,~~\nabla^0_{e_3}e_2=-\frac{\alpha}{2}e_1,~~\nabla^0_{e_3}e_3=0.
\notag
\end{align}
\end{lem}
\noindent By (2.8) and Lemma 2.6, we have
\begin{align}
&T^0(e_1,e_2)=\beta e_3,~~T^0(e_1,e_3)=(-\frac{\alpha}{2}+\beta)e_2+\gamma e_3,~~T^0(e_2,e_3)=-\frac{\alpha}{2} e_1.
\end{align}
By (2.9) and (2.29), we have
\vskip 0.5 true cm
\begin{lem}
The tensor $A^0$ of the canonical connection $\nabla^0$ of $(G_2,J)$ is given by
\begin{align}
&A^0(e_1,e_2)e_1=(\frac{\alpha\beta}{2}-\beta^2)e_2-\beta\gamma e_3,~~A^0(e_1,e_2)e_2=\frac{\alpha\beta}{2}e_1,~~A^0(e_1,e_2)e_3=0,\\\notag
&A^0(e_1,e_3)e_1=(\frac{\alpha\gamma}{2}-\beta\gamma)e_2+(\frac{\alpha\beta}{2}-\beta^2-\gamma^2)e_3
,~~A^0(e_1,e_3)e_2=\frac{\alpha\gamma}{2}e_1\\\notag
&A^0(e_1,e_3)e_3=(\frac{\alpha^2}{4}-\frac{\alpha\beta}{2})e_1,
~~A^0(e_2,e_3)e_1=0,
~~A^0(e_2,e_3)e_2=-\frac{\alpha\beta}{2}e_3,\\\notag
&A^0(e_2,e_3)e_3=(\frac{\alpha^2}{4}-\frac{\alpha\beta}{2})e_2-\frac{\alpha\gamma}{2}e_3
.\notag
\end{align}
\end{lem}
\vskip 0.5 true cm
By Lemma 2.7, (2.11) and (2.12), we have
\begin{align}
{\overline{A}}^0\left(\begin{array}{c}
e_1\\
e_2\\
e_3
\end{array}\right)=\left(\begin{array}{ccc}
-\alpha\beta+2\beta^2+\gamma^2&0&0\\
0&\alpha\beta&-\frac{\alpha\gamma}{2}\\
0&\frac{\alpha\gamma}{2}&\alpha\beta-\frac{\alpha^2}{2}
\end{array}\right)\left(\begin{array}{c}
e_1\\
e_2\\
e_3
\end{array}\right).
\end{align}
\noindent By (2.41) in \cite{Wa}, we have
\begin{align}
{\rm Ric}^0\left(\begin{array}{c}
e_1\\
e_2\\
e_3
\end{array}\right)=\left(\begin{array}{ccc}
-\left(\gamma^2+\frac{\alpha\beta}{2}\right)&0&0\\
0&-\left(\gamma^2+\frac{\alpha\beta}{2}\right)&0\\
0&\beta\gamma-\frac{\alpha\gamma}{2}&0
\end{array}\right)\left(\begin{array}{c}
e_1\\
e_2\\
e_3
\end{array}\right).
\end{align}
\noindent By (2.14), we have
\begin{align}
{{\rm Wan}}^0\left(\begin{array}{c}
e_1\\
e_2\\
e_3
\end{array}\right)=\left(\begin{array}{ccc}
\frac{\alpha\beta}{2}-2\beta^2-2\gamma^2&0&0\\
0&-(\gamma^2+\frac{3}{2}\alpha\beta)&\frac{\alpha\gamma}{2}\\
0&\beta\gamma-\alpha\gamma&\frac{\alpha^2}{2}-\alpha\beta
\end{array}\right)\left(\begin{array}{c}
e_1\\
e_2\\
e_3
\end{array}\right).
\end{align}
\noindent If $(G_2,g,J)$ is the first kind algebraic Wanas soliton associated to the connection $\nabla^0$, then
${\rm Wan}^0=c{\rm Id}+D$, so
\begin{align}
\left\{\begin{array}{l}
De_1=(\frac{\alpha\beta}{2}-2\beta^2-2\gamma^2-c)e_1,\\
De_2=-(\gamma^2+\frac{3}{2}\alpha\beta+c)e_2+\frac{\alpha\gamma}{2}e_3,\\
De_3=(\beta\gamma-\alpha\gamma)e_2+(\frac{\alpha^2}{2}-\alpha\beta-c)e_3\\
\end{array}\right.
\end{align}
Then we have
\vskip 0.5 true cm
\begin{thm}
$(G_2,g,J)$ is the first kind algebraic Wanas soliton associated to the connection $\nabla^0$ if and only if $\alpha=\beta=0$. $2\gamma^2+c=0$, $\gamma\neq 0$.
In particular,
\begin{align}
{D}\left(\begin{array}{c}
e_1\\
e_2\\
e_3
\end{array}\right)=\left(\begin{array}{ccc}
0&0&0\\
0&\gamma^2&0\\
0&0&2\gamma^2
\end{array}\right)\left(\begin{array}{c}
e_1\\
e_2\\
e_3
\end{array}\right).
\end{align}
\end{thm}
\begin{proof}
By (2.16) and (2.34), we get
\begin{align}
\left\{\begin{array}{l}
\alpha\beta+\beta^2+2\gamma^2+c=0,\\
\beta(\frac{\alpha^2}{2}+2\beta^2+3\gamma^2+c)-\alpha\gamma^2=0,\\
\beta(2\beta^2+3\gamma^2+c-\alpha\beta-\frac{\alpha^2}{2})-2\alpha\gamma^2=0,\\
\alpha(\frac{\alpha^2}{2}+2\beta^2+\gamma^2-c-3\alpha\beta)=0.\\
\end{array}\right.
\end{align}
Using the second equation and the third equation in (2.36), we have
\begin{equation}
\alpha(\gamma^2+\beta^2+\alpha\beta)=0.
\end{equation}
Using the first equation and the fourth equation in (2.36), we have
\begin{equation}
\alpha(\frac{\alpha^2}{2}+3\beta^2+3\gamma^2-2\alpha\beta)=0.
\end{equation}
If $\alpha\neq 0$, by (2.37) and (2.38), we get $\alpha=10\beta$. Again by (2.37), we get $\gamma^2+11\beta^2=0$, then $\gamma=\beta=0$. This
is a contradiction. So $\alpha=0$, by the first equation and the sencond equation in (2.36), we have
\begin{equation}
\beta^2+2\gamma^2+c=0,~~\beta(2\beta^2+3\gamma^2+c)=0.
\end{equation}
So we get $\beta=0$ and $2\gamma^2+c=0$ by $\gamma\neq 0$.
\end{proof}
By (2.12), (2.13) and (2.33), we get
\begin{align}
{\widetilde{{\rm Wan}}}^0\left(\begin{array}{c}
e_1\\
e_2\\
e_3
\end{array}\right)=\left(\begin{array}{ccc}
\frac{\alpha\beta}{2}-2\beta^2-2\gamma^2&0&0\\
0&-(\gamma^2+\frac{3}{2}\alpha\beta)&\frac{3\alpha\gamma}{4}-\frac{\beta\gamma}{2}\\
0&\frac{\beta\gamma}{2}-\frac{3\alpha\gamma}{4}&\frac{\alpha^2}{2}-\alpha\beta
\end{array}\right)\left(\begin{array}{c}
e_1\\
e_2\\
e_3
\end{array}\right).
\end{align}
\noindent If $(G_2,g,J)$ is the second kind algebraic Wanas soliton associated to the connection $\nabla^0$, then
$\widetilde{{\rm Wan}}^0=c{\rm Id}+D$, so
\begin{align}
\left\{\begin{array}{l}
De_1=(\frac{\alpha\beta}{2}-2\beta^2-2\gamma^2-c)e_1,\\
De_2=-(\gamma^2+\frac{3}{2}\alpha\beta+c)e_2+(\frac{3\alpha\gamma}{4}-\frac{\beta\gamma}{2})e_3,\\
De_3=(\frac{\beta\gamma}{2}-\frac{3\alpha\gamma}{4})e_2+(\frac{\alpha^2}{2}-\alpha\beta-c)e_3.\\
\end{array}\right.
\end{align}
\noindent By (2.16) and (2.41), we get
\begin{align}
\left\{\begin{array}{l}
\alpha\beta+\beta^2+2\gamma^2+c=0,\\
\beta(\frac{\alpha^2}{2}+2\beta^2+3\gamma^2+c)-\gamma^2(\frac{3}{2}\alpha-\beta)=0,\\
\beta(2\beta^2+2\gamma^2+c-\alpha\beta-\frac{\alpha^2}{2})-\frac{3}{2}\alpha\gamma^2=0,\\
\alpha(\frac{\alpha^2}{2}+2\beta^2+\gamma^2-c-3\alpha\beta)=0.\\
\end{array}\right.
\end{align}
\noindent Similar to Theorem 2.8, Then
\vskip 0.5 true cm
\begin{thm}
$(G_2,g,J)$ is the second kind algebraic Wanas soliton associated to the connection $\nabla^0$ if and only if $\alpha=\beta=0$. $2\gamma^2+c=0$, $\gamma\neq 0$.
In particular,
\begin{align}
{D}\left(\begin{array}{c}
e_1\\
e_2\\
e_3
\end{array}\right)=\left(\begin{array}{ccc}
0&0&0\\
0&\gamma^2&0\\
0&0&2\gamma^2
\end{array}\right)\left(\begin{array}{c}
e_1\\
e_2\\
e_3
\end{array}\right).
\end{align}
\end{thm}

\vskip 0.5 true cm
\noindent{\bf 2.3 Algebraic Wanas solitons of $G_3$}\\
\vskip 0.5 true cm
By (2.3) in \cite{BO}, we have for $G_3$, there exists a pseudo-orthonormal basis $\{e_1,e_2,e_3\}$ with $e_3$ timelike such that the Lie
algebra of $G_3$ satisfies
\begin{equation}
[e_1,e_2]=-\gamma e_3,~~[e_1,e_3]=-\beta e_2,~~[e_2,e_3]=\alpha e_1.
\end{equation}
Let
\begin{equation}
a_1=\frac{1}{2}(\alpha-\beta-\gamma),~~a_2=\frac{1}{2}(\alpha-\beta+\gamma),~~a_3=\frac{1}{2}(\alpha+\beta-\gamma).
\end{equation}
By Lemma 2.24 in \cite{Wa}, we have
\vskip 0.5 true cm
\begin{lem}
The canonical connection $\nabla^0$ of $(G_3,J)$ is given by
\begin{align}
&\nabla^0_{e_1}e_1=0,~~\nabla^0_{e_1}e_2=0,~~\nabla^0_{e_1}e_3=0,\\\notag
&\nabla^0_{e_2}e_1=0,~~\nabla^0_{e_2}e_2=0,~~\nabla^0_{e_2}e_3=0,\\\notag
&\nabla^0_{e_3}e_1=a_3e_2,~~\nabla^0_{e_3}e_2=-a_3e_1,~~\nabla^0_{e_3}e_3=0.
\notag
\end{align}
\end{lem}
\noindent By (2.8) and Lemma 2.10, we have
\begin{align}
&T^0(e_1,e_2)=\gamma e_3,~~T^0(e_1,e_3)=-a_1e_2,~~T^0(e_2,e_3)=-a_2 e_1.
\end{align}
By (2.9) and (2.47), we have
\vskip 0.5 true cm
\begin{lem}
The tensor $A^0$ of the canonical connection $\nabla^0$ of $(G_3,J)$ is given by
\begin{align}
&A^0(e_1,e_2)e_1=a_1\gamma e_2,~~A^0(e_1,e_2)e_2=a_2\gamma e_1,~~A^0(e_1,e_2)e_3=0,\\\notag
&A^0(e_1,e_3)e_1=a_1\gamma e_3
,~~A^0(e_1,e_3)e_2=0,
~~A^0(e_1,e_3)e_3=a_1a_2e_1,\\\notag
&A^0(e_2,e_3)e_1=0,
~~A^0(e_2,e_3)e_2=-a_2\gamma e_3,~~A^0(e_2,e_3)e_3=a_1a_2e_2
.\notag
\end{align}
\end{lem}
\vskip 0.5 true cm
By Lemma 2¡£11, (2.11) and (2.12), we have
\begin{align}
{\overline{A}}^0\left(\begin{array}{c}
e_1\\
e_2\\
e_3
\end{array}\right)=\left(\begin{array}{ccc}
-2a_1\gamma&0&0\\
0&2a_2 \gamma&0\\
0&0&-2a_1a_2
\end{array}\right)\left(\begin{array}{c}
e_1\\
e_2\\
e_3
\end{array}\right).
\end{align}
\noindent By (2.64) in \cite{Wa}, we have
\begin{align}
{\rm Ric}^0\left(\begin{array}{c}
e_1\\
e_2\\
e_3
\end{array}\right)=\left(\begin{array}{ccc}
-\gamma a_3&0&0\\
0&-\gamma a_3&0\\
0&0&0
\end{array}\right)\left(\begin{array}{c}
e_1\\
e_2\\
e_3
\end{array}\right).
\end{align}
\noindent By (2.14), we have
\begin{align}
{{\rm Wan}}^0\left(\begin{array}{c}
e_1\\
e_2\\
e_3
\end{array}\right)=\left(\begin{array}{ccc}
-\gamma a_3+2a_1\gamma&0&0\\
0&-\gamma a_3-2a_2\gamma&0\\
0&0&2a_1a_2
\end{array}\right)\left(\begin{array}{c}
e_1\\
e_2\\
e_3
\end{array}\right).
\end{align}
\noindent If $(G_3,g,J)$ is the first kind algebraic Wanas soliton associated to the connection $\nabla^0$, then
${\rm Wan}^0=c{\rm Id}+D$, so
\begin{align}
\left\{\begin{array}{l}
De_1=(-\gamma a_3+2a_1\gamma-c)e_1,\\
De_2=-(\gamma a_3+2a_2\gamma+c)e_2,\\
De_3=(2a_1a_2-c)e_3\\
\end{array}\right.
\end{align}
By (2.16) and (2.52), we have
\begin{align}
\left\{\begin{array}{l}
\gamma(-2\gamma a_3+2a_1\gamma-2a_2\gamma-c-2a_1a_2)=0,\\
\beta(2a_1\gamma+2a_2\gamma+2a_1a_2-c)=0,\\
\alpha(2a_1\gamma+2a_2\gamma-2a_1a_2+c)=0.\\
\end{array}\right.
\end{align}
Solving (2.53), we have
\vskip 0.5 true cm
\begin{thm}
$(G_3,g,J)$ is the first kind algebraic Wanas soliton associated to the connection $\nabla^0$ if and only if \\
\noindent (i)$\alpha\neq 0$, $\beta\neq 0$, $\gamma=0$, $c=\frac{1}{2}(\alpha-\beta)^2.$ In particular,
 \begin{align}
{D}\left(\begin{array}{c}
e_1\\
e_2\\
e_3
\end{array}\right)=\left(\begin{array}{ccc}
-\frac{1}{2}(\alpha-\beta)^2&0&0\\
0&-\frac{1}{2}(\alpha-\beta)^2&0\\
0&0&0
\end{array}\right)\left(\begin{array}{c}
e_1\\
e_2\\
e_3 \notag \\
\end{array}\right).
\end{align}
\noindent (ii)$\alpha=\beta=\gamma=0$. In particular,
 \begin{align}
{D}\left(\begin{array}{c}
e_1\\
e_2\\
e_3
\end{array}\right)=\left(\begin{array}{ccc}
-c&0&0\\
0&-c&0\\
0&0&-c
\end{array}\right)\left(\begin{array}{c}
e_1\\
e_2\\
e_3 \notag \\
\end{array}\right).
\end{align}
\noindent (iii) $\alpha=\beta=0$, $\gamma\neq 0$, $c=-\frac{1}{2}\gamma^2.$ In particular,
\begin{align}
{D}\left(\begin{array}{c}
e_1\\
e_2\\
e_3
\end{array}\right)=\left(\begin{array}{ccc}
0&0&0\\
0&0&0\\
0&0&0
\end{array}\right)\left(\begin{array}{c}
e_1\\
e_2\\
e_3 \notag \\
\end{array}\right).
\end{align}
\noindent (iv) $\alpha=\gamma=0$, $\beta\neq 0$, $c=\frac{1}{2}\beta^2$. In particular,
\begin{align}
{D}\left(\begin{array}{c}
e_1\\
e_2\\
e_3
\end{array}\right)=\left(\begin{array}{ccc}
-\frac{1}{2}\beta^2&0&0\\
0&-\frac{1}{2}\beta^2&0\\
0&0&0
\end{array}\right)\left(\begin{array}{c}
e_1\\
e_2\\
e_3 \notag \\
\end{array}\right).
\end{align}
\noindent (v) $\alpha=0$, $\beta\neq 0$, $\gamma\neq 0$, $\beta=\gamma$, $c=-2\beta^2$.
 In particular,
\begin{align}
{D}\left(\begin{array}{c}
e_1\\
e_2\\
e_3
\end{array}\right)=\left(\begin{array}{ccc}
0&0&0\\
0&2\beta^2&0\\
0&0&2\beta^2
\end{array}\right)\left(\begin{array}{c}
e_1\\
e_2\\
e_3 \notag \\
\end{array}\right).
\end{align}
\noindent (vi) $\alpha\neq 0$, $\beta=\gamma=0$, $c=\frac{1}{2}\alpha^2$.
In particular,
\begin{align}
{D}\left(\begin{array}{c}
e_1\\
e_2\\
e_3
\end{array}\right)=\left(\begin{array}{ccc}
-\frac{1}{2}\alpha^2&0&0\\
0&-\frac{1}{2}\alpha^2&0\\
0&0&0
\end{array}\right)\left(\begin{array}{c}
e_1\\
e_2\\
e_3 \notag \\
\end{array}\right).
\end{align}
\noindent (vii) $\alpha\neq 0$, $\beta=0$, $\gamma\neq 0$, $\alpha=\gamma$, $c=-2\alpha^2$.
In particular,
 \begin{align}
{D}\left(\begin{array}{c}
e_1\\
e_2\\
e_3
\end{array}\right)=\left(\begin{array}{ccc}
2\alpha^2&0&0\\
0&0&0\\
0&0&2\alpha^2
\end{array}\right)\left(\begin{array}{c}
e_1\\
e_2\\
e_3 \notag \\
\end{array}\right).
\end{align}
\end{thm}
\vskip 0.5 true cm
\indent By (2.51), we have $\widetilde{{\rm Wan}}={\rm Wan}$. So $(G_3,g,J)$ is the first kind algebraic Wanas soliton associated to the connection $\nabla^0$ if and only if $(G_3,g,J)$ is the second kind algebraic Wanas soliton associated to the connection $\nabla^0$.
\vskip 0.5 true cm
\noindent{\bf 2.4 Algebraic Wanas solitons of $G_4$}\\

By (2.4) in \cite{BO}, we have for $G_4$, there exists a pseudo-orthonormal basis $\{e_1,e_2,e_3\}$ with $e_3$ timelike such that the Lie
algebra of $G_4$ satisfies
 \begin{align}
[e_1,e_2]=-e_2+(2\eta-\beta)e_3,~~\eta=1~{\rm or}-1,~~[e_1,e_3]=-\beta e_2+ e_3,~~[e_2,e_3]=\alpha e_1.
\end{align}
Let
\begin{equation}
b_1=\frac{\alpha}{2}+\eta-\beta,~~b_2=\frac{\alpha}{2}-\eta,~~b_3=\frac{\alpha}{2}+\eta.
\end{equation}
By Lemma 2.32 in \cite{Wa}, we have
\vskip 0.5 true cm
\begin{lem}
The canonical connection $\nabla^0$ of $(G_4,J)$ is given by
\begin{align}
&\nabla^0_{e_1}e_1=0,~~\nabla^0_{e_1}e_2=0,~~\nabla^0_{e_1}e_3=0,\\\notag
&\nabla^0_{e_2}e_1= e_2,~~\nabla^0_{e_2}e_2=- e_1,~~\nabla^0_{e_2}e_3=0,\\\notag
&\nabla^0_{e_3}e_1=b_3e_2,~~\nabla^0_{e_3}e_2=-b_3e_1,~~\nabla^0_{e_3}e_3=0.
\notag
\end{align}
\end{lem}
\vskip 0.5 true cm
\noindent By (2.8) and Lemma 2.13, we have
\begin{align}
&T^0(e_1,e_2)=(\beta-2\eta) e_3,~~T^0(e_1,e_3)=-b_1e_2-e_3,~~T^0(e_2,e_3)=-b_2 e_1.
\end{align}
\noindent By (2.9) and (2.57), we have
\vskip 0.5 true cm
\begin{lem}
The tensor $A^0$ of the canonical connection $\nabla^0$ of $(G_4,J)$ is given by
\begin{align}
&A^0(e_1,e_2)e_1=(2\eta-\beta)(-b_1e_2-e_3)
,~~A^0(e_1,e_2)e_2=(\beta-2\eta)b_2e_1,
~~A^0(e_1,e_2)e_3=0,\\\notag
&A^0(e_1,e_3)e_1=-b_1e_2+[b_1(\beta-2\eta)-1]e_3
,~~A^0(e_1,e_3)e_2=-b_2e_1\\\notag
&A^0(e_1,e_3)e_3=b_1b_2e_1,
~~A^0(e_2,e_3)e_1=0,
~~A^0(e_2,e_3)e_2=b_2(2\eta-\beta)e_3,\\\notag
&A^0(e_2,e_3)e_3=b_2(b_1e_2+e_3)
.\notag
\end{align}
\end{lem}
\vskip 0.5 true cm
By Lemma 2.14, (2.11) and (2.12), we have
\begin{align}
{\overline{A}}^0\left(\begin{array}{c}
e_1\\
e_2\\
e_3
\end{array}\right)=\left(\begin{array}{ccc}
2b_1(2\eta-\beta)+1&0&0\\
0&2b_2(\beta-2\eta)&b_2\\
0&-b_2&-2b_1b_2
\end{array}\right)\left(\begin{array}{c}
e_1\\
e_2\\
e_3
\end{array}\right).
\end{align}
\noindent By (2.78) in \cite{Wa}, we have
\begin{align}
{\rm Ric}^0\left(\begin{array}{c}
e_1\\
e_2\\
e_3
\end{array}\right)=\left(\begin{array}{ccc}
b_3(2\eta-\beta)-1&0&0\\
0&b_3(2\eta-\beta)-1&0\\
0&b_3-\beta&0
\end{array}\right)\left(\begin{array}{c}
e_1\\
e_2\\
e_3
\end{array}\right).
\end{align}
\noindent By (2.14), we have
\begin{align}
{{\rm Wan}}^0\left(\begin{array}{c}
e_1\\
e_2\\
e_3
\end{array}\right)=\left(\begin{array}{ccc}
(b_3-2b_1)(2\eta-\beta)-2&0&0\\
0&(b_3+2b_2)(2\eta-\beta)-1&-b_2\\
0&b_3+b_2-\beta&2b_1b_2
\end{array}\right)\left(\begin{array}{c}
e_1\\
e_2\\
e_3
\end{array}\right).
\end{align}
\noindent If $(G_4,g,J)$ is the first kind algebraic Wanas soliton associated to the connection $\nabla^0$, then
${\rm Wan}^0=c{\rm Id}+D$, so
\begin{align}
\left\{\begin{array}{l}
De_1=[(b_3-2b_1)(2\eta-\beta)-2-c]e_1,\\
De_2=[(b_3+2b_2)(2\eta-\beta)-1-c]e_2-b_2e_3,\\
De_3=(b_3+b_2-\beta)e_2+(2b_1b_2-c)e_3.\\
\end{array}\right.
\end{align}
Then we have
\vskip 0.5 true cm
\begin{thm}
$(G_4,g,J)$ is the first kind algebraic Wanas soliton associated to the connection $\nabla^0$ if and only if\\
 \noindent (i) $\alpha=\beta=0$, $c=-4$.
In particular,
\begin{align}
{D}\left(\begin{array}{c}
e_1\\
e_2\\
e_3
\end{array}\right)=\left(\begin{array}{ccc}
0&0&0\\
0&1&\eta\\
0&0&2
\end{array}\right)\left(\begin{array}{c}
e_1\\
e_2\\
e_3 \notag \\
\end{array}\right).
\end{align}
\noindent (ii) $\alpha=0$, $\beta=\eta$, $c=-1$.
In particular,
\begin{align}
{D}\left(\begin{array}{c}
e_1\\
e_2\\
e_3
\end{array}\right)=\left(\begin{array}{ccc}
0&0&0\\
0&-1&\eta\\
0&-\eta&1
\end{array}\right)\left(\begin{array}{c}
e_1\\
e_2\\
e_3 \notag \\
\end{array}\right).
\end{align}
\end{thm}
\begin{proof}
By (2.16) and (2.62), we get
\begin{align}
\left\{\begin{array}{l}
(2\eta-\beta)(-2b_1+b_2+2b_3-\beta)-c-2-b_2\beta=0,\\
(2\eta-\beta)[(2b_2+2b_3-2b_1)(2\eta-\beta)-3-c-2b_1b_2]-2b_2=0,\\
\beta[(2b_1+2b_2)(2\eta-\beta)+1+c-2b_1b_2]-2(b_3+b_2-\beta)=0,\\
\alpha[(2b_1+2b_2)(2\eta-\beta)+1-c+2b_1b_2]=0.\\
\end{array}\right.
\end{align}
\noindent Case (i) $\alpha=0$. Then $b_1=\eta-\beta$, $b_2=-\eta$, $b_3=\eta$. By (2.63), we get
\begin{align}
\left\{\begin{array}{l}
(2\eta-\beta)(\beta-\eta)-c-2+\eta\beta=0,\\
(2\eta-\beta)[-2(\eta-\beta)^2-3-c]+2\eta=0,\\
\beta[-2\beta(2\eta-\beta)+3+c+2(\eta-\beta)\eta]=0.\\
\end{array}\right.
\end{align}
By the first and third equations in (2.64), we can get $\beta(\beta-\eta)^2=0$, then $\beta=0$ or $\beta=\eta$.
When $\beta=0$, we get $c=-4$ by (2.64) which is (i) in Theorem 2.15. When $\beta=\eta$, we get $c=-1$ by (2.64) which is (ii) in Theorem 2.15.\\
\noindent Case (ii) $\alpha\neq 0$. By the fourth equation in (2.63), we get
\begin{equation}(2b_1+2b_2)(2\eta-\beta)+1=c-2b_1b_2.
\end{equation}
By the third equation in (2.63) and (2.65), we get
\begin{equation}\alpha=\beta+\frac{\beta}{1-2\beta(2\eta-\beta)}.
\end{equation}
${\rm The ~~second~~ equation ~~in~~ (2.63)}-2(2\eta-\beta)\times{\rm the ~~first ~~equation~~ in~~ (2.63)}$, then we get
\begin{equation}
(2\eta-\beta)[(2\eta-\beta)(4b_1+2b_2-2b_3+2\beta)+2+2b_2\beta]-2b_2=0.
\end{equation}
Then
\begin{equation}
\alpha[2(2\eta-\beta)^2+\beta(2\eta-\beta)-1]=2\beta(2\eta-\beta)^2+2\eta\beta(2\eta-\beta)-2(2\eta-\beta)-2\eta.
\end{equation}
By (2.66) and (2.68), we get
\begin{equation}
\beta^5-6\eta\beta^4+15\beta^3-19\eta\beta^2+12\beta-3\eta=0.
\end{equation}
When $\eta=1$, we get $(\beta-1)^3(\beta^2-3\beta+3)=0.$ Then $\beta=1$. By (2.66), $\alpha=0$. This is a contradiction. When $\eta=-1$, we get $(\beta+1)^3(\beta^2+3\beta+3)=0.$ Then $\beta=-1$. By (2.66), $\alpha=0$. This is a contradiction. So in this case, we have no solutions.
\end{proof}
By (2.12), (2.13) and (2.61), we get
\begin{align}
{\widetilde{{\rm Wan}}}^0\left(\begin{array}{c}
e_1\\
e_2\\
e_3
\end{array}\right)==\left(\begin{array}{ccc}
(b_3-2b_1)(2\eta-\beta)-2&0&0\\
0&(b_3+2b_2)(2\eta-\beta)-1&-\frac{b_3+2b_2-\beta}{2}\\
0&\frac{b_3+2b_2-\beta}{2}&2b_1b_2
\end{array}\right)\left(\begin{array}{c}
e_1\\
e_2\\
e_3
\end{array}\right).
\end{align}
\noindent If $(G_4,g,J)$ is the second kind algebraic Wanas soliton associated to the connection $\nabla^0$, then
$\widetilde{{\rm Wan}}^0=c{\rm Id}+D$, so
\begin{align}
\left\{\begin{array}{l}
De_1=[(b_3-2b_1)(2\eta-\beta)-2-c]e_1,\\
De_2=[(b_3+2b_2)(2\eta-\beta)-1-c]e_2-\frac{b_3+2b_2-\beta}{2}e_3,\\
De_3=\frac{b_3+2b_2-\beta}{2}e_2+(2b_1b_2-c)e_3.\\
\end{array}\right.
\end{align}

\noindent By (2.16) and (2.71), we get
\begin{align}
\left\{\begin{array}{l}
(2\eta-\beta)(-2b_1+b_2+\frac{3}{2}b_3-\frac{\beta}{2})-c-2-\frac{b_3+2b_2-\beta}{2}\beta=0,\\
(2\eta-\beta)[(2b_2+2b_3-2b_1)(2\eta-\beta)-3-c-2b_1b_2]-(b_3+2b_2-\beta)=0,\\
\beta[(2b_1+2b_2)(2\eta-\beta)+1+c-2b_1b_2]-(b_3+2b_2-\beta)=0,\\
\alpha[(2b_1+2b_2)(2\eta-\beta)+1-c+2b_1b_2]=0.\\
\end{array}\right.
\end{align}
Solving (2.72), similar to Theorem 2.15, we get

\vskip 0.5 true cm
\begin{thm}
$(G_4,g,J)$ is the second kind algebraic Wanas soliton associated to the connection $\nabla^0$ if and only if\\
 \noindent (i) $\alpha=0$, $\beta=\eta$, $c=-1$.
In particular,
\begin{align}
{D}\left(\begin{array}{c}
e_1\\
e_2\\
e_3
\end{array}\right)=\left(\begin{array}{ccc}
0&0&0\\
0&-1&\eta\\
0&-\eta&1
\end{array}\right)\left(\begin{array}{c}
e_1\\
e_2\\
e_3 \notag \\
\end{array}\right).
\end{align}
\noindent (ii) $\alpha=0$, $\beta=-\eta$, $c=-11$.
In particular,
\begin{align}
{D}\left(\begin{array}{c}
e_1\\
e_2\\
e_3
\end{array}\right)=\left(\begin{array}{ccc}
0&0&0\\
0&7&0\\
0&0&7
\end{array}\right)\left(\begin{array}{c}
e_1\\
e_2\\
e_3 \notag \\
\end{array}\right).
\end{align}
\end{thm}
\vskip 0.5 true cm

\section{ Algebraic Wanas solitons associated to canonical connections on three-dimensional non-unimodular Lorentzian Lie groups}
\vskip 0.5 true cm
\noindent{\bf 3.1 Algebraic Wanas solitons of $G_5$}
\vskip 0.5 true cm
By (2.5) in \cite{BO}, we have for $G_5$, there exists a pseudo-orthonormal basis $\{e_1,e_2,e_3\}$ with $e_3$ timelike such that the Lie
algebra of $G_5$ satisfies
\begin{equation}
[e_1,e_2]=0,~~[e_1,e_3]=\alpha e_1+\beta e_2,~~[e_2,e_3]=\gamma e_1+\delta e_2,~~\alpha+\delta\neq 0,~~\alpha\gamma+\beta\delta=0.
\end{equation}
By Lemma 3.3 in \cite{Wa}, we have
\vskip 0.5 true cm
\begin{lem}
The canonical connection $\nabla^0$ of $(G_5,J)$ is given by
\begin{align}
&\nabla^0_{e_1}e_1=0,~~\nabla^0_{e_1}e_2=0,~~\nabla^0_{e_1}e_3=0,\\\notag
&\nabla^0_{e_2}e_1=0,~~\nabla^0_{e_2}e_2=0,~~\nabla^0_{e_2}e_3=0,\\\notag
&\nabla^0_{e_3}e_1=-\frac{\beta-\gamma}{2}e_2,~~\nabla^0_{e_3}e_2=\frac{\beta-\gamma}{2}e_1,~~\nabla^0_{e_3}e_3=0.
\notag
\end{align}
\end{lem}
\vskip 0.5 true cm
\noindent By (2.8) and Lemma 3.1, we have
\begin{align}
&T^0(e_1,e_2)=0,~~T^0(e_1,e_3)=-\alpha e_1-\frac{\beta+\gamma}{2}e_2,~~T^0(e_2,e_3)=-\frac{\beta+\gamma}{2} e_1-\delta e_2.
\end{align}
\noindent By (2.9) and (3.3), we have
\vskip 0.5 true cm
\begin{lem}
The tensor $A^0$ of the canonical connection $\nabla^0$ of $(G_5,J)$ is given by
\begin{align}
&A^0(e_1,e_2)e_1=0,~~A^0(e_1,e_2)e_2=0,~~A^0(e_1,e_2)e_3=0,\\\notag
&A^0(e_1,e_3)e_1=0
,~~A^0(e_1,e_3)e_2=0,~~
A^0(e_1,e_3)e_3=[\alpha^2+\frac{(\beta+\gamma)^2}{4}]e_1+\frac{\beta+\gamma}{2}(\alpha+\delta)e_2
,\\\notag
&A^0(e_2,e_3)e_1=0,
~~A^0(e_2,e_3)e_2=0,~~
A^0(e_2,e_3)e_3=\frac{\beta+\gamma}{2}(\alpha+\delta)e_1+[\frac{(\beta+\gamma)^2}{4}+\delta^2]e_2
.\notag
\end{align}
\end{lem}
\vskip 0.5 true cm
By Lemma 3.2, (2.11) and (2.12), we have
\begin{align}
{\overline{A}}^0\left(\begin{array}{c}
e_1\\
e_2\\
e_3
\end{array}\right)=\left(\begin{array}{ccc}
0&0&0\\
0&0&0\\
0&0&-[\alpha^2+\frac{(\beta+\gamma)^2}{2}+\delta^2]
\end{array}\right)\left(\begin{array}{c}
e_1\\
e_2\\
e_3
\end{array}\right).
\end{align}
\noindent By (3.5) in \cite{Wa}, we have
\begin{align}
{\rm Ric}^0\left(\begin{array}{c}
e_1\\
e_2\\
e_3
\end{array}\right)=\left(\begin{array}{ccc}
0&0&0\\
0&0&0\\
0&0&0
\end{array}\right)\left(\begin{array}{c}
e_1\\
e_2\\
e_3
\end{array}\right).
\end{align}
\noindent By (2.14), we have
\begin{align}
{{\rm Wan}}^0\left(\begin{array}{c}
e_1\\
e_2\\
e_3
\end{array}\right)=\left(\begin{array}{ccc}
0&0&0\\
0&0&0\\
0&0&[\alpha^2+\frac{(\beta+\gamma)^2}{2}+\delta^2]
\end{array}\right)\left(\begin{array}{c}
e_1\\
e_2\\
e_3
\end{array}\right).
\end{align}
\noindent If $(G_5,g,J)$ is the first kind algebraic Wanas soliton associated to the connection $\nabla^0$, then
${\rm Wan}^0=c{\rm Id}+D$, so
\begin{align}
\left\{\begin{array}{l}
De_1=-ce_1,\\
De_2=-ce_2,\\
De_3=[\alpha^2+\frac{(\beta+\gamma)^2}{2}+\delta^2-c]e_3.\\
\end{array}\right.
\end{align}
By (2.16) and (3.8), we get
\begin{align}
\left\{\begin{array}{l}
\alpha[\alpha^2+\frac{(\beta+\gamma)^2}{2}+\delta^2-c]=0,\\
\beta[\alpha^2+\frac{(\beta+\gamma)^2}{2}+\delta^2-c]=0,\\
\gamma[\alpha^2+\frac{(\beta+\gamma)^2}{2}+\delta^2-c]=0,\\
\delta[\alpha^2+\frac{(\beta+\gamma)^2}{2}+\delta^2-c]=0,\\
\end{array}\right.
\end{align}
Solving (3.9), we get
\vskip 0.5 true cm
\begin{thm}
$(G_5,g,J)$ is the first kind algebraic Wanas soliton associated to the connection $\nabla^0$ if and only if $\alpha^2+\frac{(\beta+\gamma)^2}{2}+\delta^2=c$.
In particular,
\begin{align}
{D}\left(\begin{array}{c}
e_1\\
e_2\\
e_3
\end{array}\right)=\left(\begin{array}{ccc}
-[\alpha^2+\frac{(\beta+\gamma)^2}{2}+\delta^2]&0&0\\
0&-[\alpha^2+\frac{(\beta+\gamma)^2}{2}+\delta^2]&0\\
0&0&0
\end{array}\right)\left(\begin{array}{c}
e_1\\
e_2\\
e_3
\end{array}\right).
\end{align}
\end{thm}
\vskip 0.5 true cm
\indent By (3.7), we have $\widetilde{{\rm Wan}}={\rm Wan}$ for $G_5$. So $(G_5,g,J)$ is the first kind algebraic Wanas soliton associated to the connection $\nabla^0$ if and only if $(G_5,g,J)$ is the second kind algebraic Wanas soliton associated to the connection $\nabla^0$.

\vskip 0.5 true cm
\noindent{\bf 3.2 Algebraic Wanas solitons of $G_6$}\\
\vskip 0.5 true cm
By (2.6) in \cite{BO}, we have for $G_6$, there exists a pseudo-orthonormal basis $\{e_1,e_2,e_3\}$ with $e_3$ timelike such that the Lie
algebra of $G_6$ satisfies
\begin{equation}
[e_1,e_2]=\alpha e_2+\beta e_3,~~[e_1,e_3]=\gamma e_2+\delta e_3,~~[e_2,e_3]=0,~~\alpha+\delta\neq 0£¬~~\alpha\gamma-\beta\delta=0.
\end{equation}
By Lemma 3.11 in \cite{Wa}, we have
\vskip 0.5 true cm
\begin{lem}
The canonical connection $\nabla^0$ of $(G_6,J)$ is given by
\begin{align}
&\nabla^0_{e_1}e_1=0,~~\nabla^0_{e_1}e_2=0,~~\nabla^0_{e_1}e_3=0,\\\notag
&\nabla^0_{e_2}e_1=-\alpha e_2,~~\nabla^0_{e_2}e_2=\alpha e_1,~~\nabla^0_{e_2}e_3=0,\\\notag
&\nabla^0_{e_3}e_1=\frac{\beta-\gamma}{2}e_2,~~\nabla^0_{e_3}e_2=-\frac{\beta-\gamma}{2}e_1,~~\nabla^0_{e_3}e_3=0.
\notag
\end{align}
\end{lem}
\noindent By (2.8) and Lemma 3.4, we have
\begin{align}
&T^0(e_1,e_2)=-\beta e_3,~~T^0(e_1,e_3)=-\frac{\beta+\gamma}{2}e_2-\delta e_3,~~T^0(e_2,e_3)=\frac{\beta-\gamma}{2} e_1.
\end{align}
\noindent By (2.9) and (3.13), we have
\vskip 0.5 true cm
\begin{lem}
The tensor $A^0$ of the canonical connection $\nabla^0$ of $(G_6,J)$ is given by
\begin{align}
&A^0(e_1,e_2)e_1=-\frac{\beta(\beta+\gamma)}{2}e_2-\beta\delta e_3,~~
A^0(e_1,e_2)e_2=\frac{\beta(\beta-\gamma)}{2}e_1
,~~A^0(e_1,e_2)e_3=0,\\\notag
&A^0(e_1,e_3)e_1=-\frac{\delta(\beta+\gamma)}{2}e_2-[\frac{\beta(\beta+\gamma)}{2}+\delta^2]e_3
,~~A^0(e_1,e_3)e_2=\frac{\delta(\beta-\gamma)}{2}e_1\\\notag
&A^0(e_1,e_3)e_3=\frac{\gamma^2-\beta^2}{4}e_1,
~~A^0(e_2,e_3)e_1=0,
~~A^0(e_2,e_3)e_2=-\frac{\beta(\beta-\gamma)}{2}e_3,\\\notag
&A^0(e_2,e_3)e_3=\frac{\gamma^2-\beta^2}{4}e_2-\frac{\delta(\beta-\gamma)}{2}e_3
.\notag
\end{align}
\end{lem}
\vskip 0.5 true cm
By Lemma 3.5, (2.11) and (2.12), we have
\begin{align}
{\overline{A}}^0\left(\begin{array}{c}
e_1\\
e_2\\
e_3
\end{array}\right)=\left(\begin{array}{ccc}
\beta^2+\beta\gamma+\delta^2&0&0\\
0&\beta(\beta-\gamma)&\frac{\delta(\gamma-\beta)}{2}\\
0&\frac{\delta(\beta-\gamma)}{2}&\frac{\beta^2-\gamma^2}{2}
\end{array}\right)\left(\begin{array}{c}
e_1\\
e_2\\
e_3
\end{array}\right).
\end{align}
\noindent By (3.15) in \cite{Wa}, we have
\begin{align}
{\rm Ric}^0\left(\begin{array}{c}
e_1\\
e_2\\
e_3
\end{array}\right)=\left(\begin{array}{ccc}
\frac{1}{2}\beta(\beta-\gamma)-\alpha^2&0&0\\
0&\frac{1}{2}\beta(\beta-\gamma)-\alpha^2&0\\
0&-\gamma\alpha+\frac{1}{2}\delta(\beta-\gamma)&0
\end{array}\right)\left(\begin{array}{c}
e_1\\
e_2\\
e_3
\end{array}\right).
\end{align}
\noindent By (2.14), we have
\begin{align}
{{\rm Wan}}^0\left(\begin{array}{c}
e_1\\
e_2\\
e_3
\end{array}\right)=\left(\begin{array}{ccc}
-(\alpha^2+\frac{\beta^2}{2}+\delta^2+\frac{3}{2}\beta\gamma)&0&0\\
0&-[\frac{1}{2}\beta(\beta-\gamma)+\alpha^2]&\frac{\delta(\beta-\gamma)}{2}\\
0&-\gamma\alpha&\frac{\gamma^2-\beta^2}{2}
\end{array}\right)\left(\begin{array}{c}
e_1\\
e_2\\
e_3
\end{array}\right).
\end{align}
\noindent If $(G_6,g,J)$ is the first kind algebraic Wanas soliton associated to the connection $\nabla^0$, then
${\rm Wan}^0=c{\rm Id}+D$, so
\begin{align}
\left\{\begin{array}{l}
De_1=-(\alpha^2+\frac{\beta^2}{2}+\delta^2+\frac{3}{2}\beta\gamma+c)e_1,\\
De_2=-[\frac{1}{2}\beta(\beta-\gamma)+\alpha^2+c]e_2+\frac{\delta(\beta-\gamma)}{2} e_3,\\
De_3=-\gamma\alpha e_2+(\frac{\gamma^2-\beta^2}{2}-c)e_3.\\
\end{array}\right.
\end{align}
\noindent Then we have
\vskip 0.5 true cm
\begin{thm}
$(G_6,g,J)$ is the first kind algebraic Wanas soliton associated to the connection $\nabla^0$ if and only if\\
 \noindent (i)$\alpha\neq 0$. $\beta=\gamma=\delta=0$, $\alpha^2+c=0$.
In particular,
\begin{align}
{D}\left(\begin{array}{c}
e_1\\
e_2\\
e_3
\end{array}\right)=\left(\begin{array}{ccc}
0&0&0\\
0&0&0\\
0&0&\alpha^2
\end{array}\right)\left(\begin{array}{c}
e_1\\
e_2\\
e_3 \notag \\
\end{array}\right).
\end{align}
\noindent (ii)$\beta=\gamma=0$, $\alpha^2+\delta^2+c=0$, $\delta\neq 0$, $\alpha+\delta\neq 0$.
In particular,
\begin{align}
{D}\left(\begin{array}{c}
e_1\\
e_2\\
e_3
\end{array}\right)=\left(\begin{array}{ccc}
0&0&0\\
0&\delta^2&0\\
0&0&\alpha^2+\delta^2
\end{array}\right)\left(\begin{array}{c}
e_1\\
e_2\\
e_3 \notag \\
\end{array}\right).
\end{align}
\noindent (iii) $\beta\neq 0$, $\delta\neq 0$, $\gamma=-\beta$, $\delta=-\alpha$, $\alpha^2=\beta^2$, $\alpha^2+c=0$.
In particular,
\begin{align}
{D}\left(\begin{array}{c}
e_1\\
e_2\\
e_3
\end{array}\right)=\left(\begin{array}{ccc}
0&0&0\\
0&-\beta^2&-\alpha\beta\\
0&\alpha\beta&\alpha^2
\end{array}\right)\left(\begin{array}{c}
e_1\\
e_2\\
e_3 \notag \\
\end{array}\right).
\end{align}
\end{thm}
\begin{proof}
By (2.16) and (3.18), we get
\begin{align}
\left\{\begin{array}{l}
\alpha(\alpha^2+\frac{\beta^2}{2}+\delta^2+\frac{3}{2}\beta\gamma+c)-\alpha\beta\gamma-\frac{1}{2}\gamma\delta(\beta-\gamma)=0,\\
\beta(2\alpha^2+\frac{\beta^2}{2}+\frac{\gamma^2}{2}+\delta^2+\beta\gamma+c)+\frac{1}{2}\delta(\beta-\gamma)(\alpha-\delta)=0,\\
\gamma(\frac{\beta^2}{2}-\frac{\gamma^2}{2}+\delta^2+2\beta\gamma+c)+(\alpha-\delta)\gamma\alpha=0,\\
\delta(\alpha^2+\frac{\beta^2}{2}+\delta^2+\frac{3}{2}\beta\gamma+c)+\alpha\beta\gamma+\frac{1}{2}\gamma\delta(\beta-\gamma)=0.\\
\end{array}\right.
\end{align}
Using the first equation and the fourth equation in (3.19) and $\alpha+\delta\neq 0$, we have
\begin{align}
\left\{\begin{array}{l}
\alpha^2+\frac{\beta^2}{2}+\delta^2+\frac{3}{2}\beta\gamma+c=0,\\
\alpha\beta\gamma+\frac{1}{2}\gamma\delta(\beta-\gamma)=0,\\
\end{array}\right.
\end{align}
By the second equation in (3.20) and $\alpha\gamma-\beta\delta=0$, we get
\begin{equation}
\delta(2\beta-\gamma)(\beta+\gamma)=0.
\end{equation}
\noindent Case 1) $\delta=0$. By $\alpha\gamma-\beta\delta=0$ and $\alpha+\delta\neq 0$, we get $\alpha\neq 0$ and $\gamma=0$.
By the first equation in (3.20) and  the second equation in (3.19), we get $\alpha^2+\frac{\beta^2}{2}+c=0$ and $\beta(2\alpha^2+\frac{\beta^2}{2}+c)=0$.
So $\beta=0$ and $\alpha^2+c=0$. This is the case (i) in Theorem 3.6.\\
\noindent Case 2)-a) $\delta\neq 0$ and $\beta=0$. By (3.21), then $\gamma=0$. By the first equation in (3.20), we have
$\alpha^2+\delta^2+c=0$. This is the case (ii) in Theorem 3.6.\\
\noindent Case 2)-b) $\delta\neq 0$ and $\beta\neq 0$. By (3.21), then $\gamma=2\beta$ or $\gamma=-\beta$.\\
\noindent Case 2)-b)-i) $\delta\neq 0$ and $\beta\neq 0$ and $\gamma=2\beta$. By $\alpha\gamma-\beta\delta=0$, then $\delta=2\alpha$.
By the first equation in (3.20) and  the second equation in (3.19), we get
\begin{equation}
\beta(\alpha^2+\frac{\gamma^2}{2}-\frac{1}{2}\beta\gamma)+\frac{1}{2}\delta(\beta-\gamma)(\alpha-\delta)=0.
\end{equation}
By $\gamma=2\beta$ and $\delta=2\alpha$ and (3.22), we get $\beta(2\alpha^2+\beta^2)=0$, then $\alpha=\beta=0$. This is a contradiction. So in this case
we have no solutions.\\
\noindent Case 2)-b)-ii) $\delta\neq 0$ and $\beta\neq 0$ and $\gamma=-\beta$. By $\alpha\gamma-\beta\delta=0$, then $\delta=-\alpha$.
By (3.19) and (3.20), we get $\alpha^2=\beta^2$ and $\alpha^2+c=0$. This is the case (iii) in Theorem 3.6.\\
\end{proof}
By (2.12), (2.13) and (3.17), we get
\begin{align}
{\widetilde{{\rm Wan}}}^0\left(\begin{array}{c}
e_1\\
e_2\\
e_3
\end{array}\right)=\left(\begin{array}{ccc}
-(\alpha^2+\frac{\beta^2}{2}+\delta^2+\frac{3}{2}\beta\gamma)&0&0\\
0&-[\frac{1}{2}\beta(\beta-\gamma)+\alpha^2]&\frac{\gamma\alpha}{2}+\frac{\delta(\beta-\gamma)}{4}\\
0&-[\frac{\gamma\alpha}{2}+\frac{\delta(\beta-\gamma)}{4}]&\frac{\gamma^2-\beta^2}{2}
\end{array}\right)\left(\begin{array}{c}
e_1\\
e_2\\
e_3
\end{array}\right).
\end{align}
\noindent If $(G_6,g,J)$ is the second kind algebraic Wanas soliton associated to the connection $\nabla^0$, then
$\widetilde{{\rm Wan}}^0=c{\rm Id}+D$, so
\begin{align}
\left\{\begin{array}{l}
De_1=-(\alpha^2+\frac{\beta^2}{2}+\delta^2+\frac{3}{2}\beta\gamma+c)e_1,\\
De_2=-[\frac{1}{2}\beta(\beta-\gamma)+\alpha^2+c]e_2+[\frac{\gamma\alpha}{2}+\frac{\delta(\beta-\gamma)}{4}]e_3,\\
De_3=-[\frac{\gamma\alpha}{2}+\frac{\delta(\beta-\gamma)}{4}]e_2+(\frac{\gamma^2-\beta^2}{2}-c)e_3.\\
\end{array}\right.
\end{align}
\noindent By (2.16) and (3.24), we get
\begin{align}
\left\{\begin{array}{l}
\alpha(\alpha^2+\frac{\beta^2}{2}+\delta^2+\frac{3}{2}\beta\gamma+c)-(\beta+\gamma)[\frac{\gamma\alpha}{2}+\frac{\delta(\beta-\gamma)}{4}]=0,\\
\beta(2\alpha^2+\frac{\beta^2}{2}+\frac{\gamma^2}{2}+\delta^2+\beta\gamma+c)+(\alpha-\delta)[\frac{\gamma\alpha}{2}+\frac{\delta(\beta-\gamma)}{4}]=0,\\
\gamma(\frac{\beta^2}{2}-\frac{\gamma^2}{2}+\delta^2+2\beta\gamma+c)+(\alpha-\delta)[\frac{\gamma\alpha}{2}+\frac{\delta(\beta-\gamma)}{4}]=0,\\
\delta(\alpha^2+\frac{\beta^2}{2}+\delta^2+\frac{3}{2}\beta\gamma+c)+(\beta+\gamma)[\frac{\gamma\alpha}{2}+\frac{\delta(\beta-\gamma)}{4}]=0.\\
\end{array}\right.
\end{align}
Solving (3.25) and  similar to Theorem 3.6, we have
\vskip 0.5 true cm
\begin{thm}
$(G_6,g,J)$ is the second kind algebraic Wanas soliton associated to the connection $\nabla^0$ if and only if\\
 \noindent (i)$\alpha\neq 0$. $\beta=\gamma=\delta=0$, $\alpha^2+c=0$.
In particular,
\begin{align}
{D}\left(\begin{array}{c}
e_1\\
e_2\\
e_3
\end{array}\right)=\left(\begin{array}{ccc}
0&0&0\\
0&0&0\\
0&0&\alpha^2
\end{array}\right)\left(\begin{array}{c}
e_1\\
e_2\\
e_3 \notag \\
\end{array}\right).
\end{align}
\noindent (ii)$\beta=\gamma=0$, $\alpha^2+\delta^2+c=0$, $\delta\neq 0$, $\alpha+\delta\neq 0$.
In particular,
\begin{align}
{D}\left(\begin{array}{c}
e_1\\
e_2\\
e_3
\end{array}\right)=\left(\begin{array}{ccc}
0&0&0\\
0&\delta^2&0\\
0&0&\alpha^2+\delta^2
\end{array}\right)\left(\begin{array}{c}
e_1\\
e_2\\
e_3 \notag \\
\end{array}\right).
\end{align}
\noindent (iii) $\beta\neq 0$, $\delta\neq 0$, $\gamma=-\beta$, $\delta=-\alpha$, $\alpha^2=\beta^2$, $\alpha^2+c=0$.
In particular,
\begin{align}
{D}\left(\begin{array}{c}
e_1\\
e_2\\
e_3
\end{array}\right)=\left(\begin{array}{ccc}
0&0&0\\
0&-\alpha^2&-\alpha\beta\\
0&\alpha\beta&\alpha^2
\end{array}\right)\left(\begin{array}{c}
e_1\\
e_2\\
e_3 \notag \\
\end{array}\right).
\end{align}
\end{thm}
\vskip 0.5 true cm
\noindent{\bf 3.3 Algebraic Wanas solitons of $G_7$}\\
\vskip 0.5 true cm
By (2.7) in \cite{BO}, we have for $G_7$, there exists a pseudo-orthonormal basis $\{e_1,e_2,e_3\}$ with $e_3$ timelike such that the Lie
algebra of $G_7$ satisfies
\begin{equation}
[e_1,e_2]=-\alpha e_1-\beta e_2-\beta e_3,~~[e_1,e_3]=\alpha e_1+\beta e_2+\beta e_3,~~[e_2,e_3]=\gamma e_1+\delta e_2+\delta e_3,,~~\alpha+\delta\neq 0,~~\alpha\gamma=0.
\end{equation}
By Lemma 3.20 in \cite{Wa}, we have
\vskip 0.5 true cm
\begin{lem}
The canonical connection $\nabla^0$ of $(G_7,J)$ is given by
\begin{align}
&\nabla^0_{e_1}e_1=\alpha e_2,~~\nabla^0_{e_1}e_2=-\alpha e_1,~~\nabla^0_{e_1}e_3=0,\\\notag
&\nabla^0_{e_2}e_1=\beta e_2,~~\nabla^0_{e_2}e_2=-\beta e_1,~~\nabla^0_{e_2}e_3=0,\\\notag
&\nabla^0_{e_3}e_1=-(\beta-\frac{\gamma}{2})e_2,~~\nabla^0_{e_3}e_2=(\beta-\frac{\gamma}{2})e_1,~~\nabla^0_{e_3}e_3=0.
\notag
\end{align}
\end{lem}
\vskip 0.5 true cm
\noindent By (2.8) and Lemma 3.8, we have
\begin{align}
&T^0(e_1,e_2)=\beta e_3,~~T^0(e_1,e_3)=-\alpha e_1-\frac{\gamma}{2}e_2-\beta e_3,~~T^0(e_2,e_3)=-(\beta+\frac{\gamma}{2})e_1-\delta e_2 -\delta e_3.
\end{align}
By (2.9) and (3.28), we have
\vskip 0.5 true cm
\begin{lem}
The tensor $A^0$ of the canonical connection $\nabla^0$ of $(G_7,J)$ is given by
\begin{align}
&A^0(e_1,e_2)e_1=\alpha\beta e_1+\frac{\beta\gamma}{2}e_2+\beta^2e_3,~~
A^0(e_1,e_2)e_2=\beta(\beta+\frac{\gamma}{2})e_1+\beta\delta e_2+\beta\delta e_3,\\\notag
&A^0(e_1,e_2)e_3=0,~~
A^0(e_1,e_3)e_1=-\alpha\beta e_1-\frac{\beta\gamma}{2}e_2+(\frac{\beta\gamma}{2}-\beta^2)e_3,\\\notag
&A^0(e_1,e_3)e_2=-\beta(\beta+\frac{\gamma}{2})e_1-\beta\delta e_2-(\alpha\beta+\beta\delta)e_3,~~\\\notag
&A^0(e_1,e_3)e_3=(\alpha^2+\frac{\beta\gamma}{2}+\frac{\gamma^2}{4})e_1+(\frac{\alpha\gamma}{2}+\frac{\gamma\delta}{2})e_2
+(\alpha\beta+\frac{\gamma\delta}{2})e_3,\\\notag
&A^0(e_2,e_3)e_1=-\alpha\delta e_1-\frac{\gamma\delta}{2}e_2,\\\notag
&A^0(e_2,e_3)e_2=-\delta(\beta+\frac{\gamma}{2})e_1-\delta^2 e_2-(\delta^2+\beta^2+\frac{\beta\gamma}{2})e_3,~~\\\notag
&A^0(e_2,e_3)e_3=(\alpha+\delta)(\beta+\frac{\gamma}{2})e_1+(\frac{\beta\gamma}{2}+\frac{\gamma^2}{4}+\delta^2)e_2
+(\delta^2+\beta^2+\frac{\beta\gamma}{2})e_3
.\notag
\end{align}
\end{lem}
\vskip 0.5 true cm
By Lemma 3.9, (2.11) and (2.12), we have
\begin{align}
{\overline{A}}^0\left(\begin{array}{c}
e_1\\
e_2\\
e_3
\end{array}\right)=\left(\begin{array}{ccc}
\beta^2-\beta\gamma&\alpha\beta&\alpha\beta+\frac{\gamma\delta}{2}\\
\alpha\beta&\delta^2+2\beta^2+\beta\gamma&\beta^2+\frac{\beta\gamma}{2}+\delta^2\\
-(\alpha\beta+\frac{\gamma\delta}{2})&-(\beta^2+\frac{\beta\gamma}{2}+\delta^2)&-(\alpha^2+\beta\gamma+\frac{\gamma^2}{2}+\delta^2)
\end{array}\right)\left(\begin{array}{c}
e_1\\
e_2\\
e_3
\end{array}\right).
\end{align}
\noindent By (3.31) in \cite{Wa}, we have
\begin{align}
{\rm Ric}^0\left(\begin{array}{c}
e_1\\
e_2\\
e_3
\end{array}\right)=\left(\begin{array}{ccc}
-(\alpha^2+\frac{\beta\gamma}{2})&0&0\\
0&-(\alpha^2+\frac{\beta\gamma}{2})&0\\
-(\gamma\alpha+\frac{\delta\gamma}{2})&\alpha^2+\frac{\beta\gamma}{2}&0
\end{array}\right)\left(\begin{array}{c}
e_1\\
e_2\\
e_3
\end{array}\right).
\end{align}
\noindent By (2.14), we have
\begin{align}
{{\rm Wan}}^0\left(\begin{array}{c}
e_1\\
e_2\\
e_3
\end{array}\right)=\left(\begin{array}{ccc}
-(\alpha^2-\frac{\beta\gamma}{2}+\beta^2)&-\alpha\beta&-(\alpha\beta+\frac{\gamma\delta}{2})\\
-\alpha\beta&-(\alpha^2+2\beta^2+\delta^2+\frac{3}{2}\beta\gamma)&-(\beta^2+\frac{\beta\gamma}{2}+\delta^2)\\
\alpha\beta-\gamma\alpha&\alpha^2+\beta^2+\delta^2+\beta\gamma&\alpha^2+\beta\gamma+\frac{\gamma^2}{2}+\delta^2
\end{array}\right)\left(\begin{array}{c}
e_1\\
e_2\\
e_3
\end{array}\right).
\end{align}
\noindent If $(G_7,g,J)$ is the first kind algebraic Wanas soliton associated to the connection $\nabla^0$, then
${\rm Wan}^0=c{\rm Id}+D$, so
\begin{align}
\left\{\begin{array}{l}
De_1=-(\alpha^2-\frac{\beta\gamma}{2}+\beta^2+c)e_1-\alpha\beta e_2-(\alpha\beta+\frac{\gamma\delta}{2})e_3,\\
De_2=-\alpha\beta e_1-(\alpha^2+2\beta^2+\delta^2+\frac{3}{2}\beta\gamma+c)e_2-(\beta^2+\frac{\beta\gamma}{2}+\delta^2)e_3,\\
De_3=(\alpha\beta-\gamma\alpha)e_1+(\alpha^2+\beta^2+\delta^2+\beta\gamma)e_2+(\alpha^2+\beta\gamma+\frac{\gamma^2}{2}+\delta^2-c)e_3.\\
\end{array}\right.
\end{align}
\noindent Then we have
\vskip 0.5 true cm
\begin{thm}
$(G_7,g,J)$ is the first kind algebraic Wanas soliton associated to the connection $\nabla^0$ if and only if $\alpha=\gamma=0$, $\beta^2+c=0$, $\delta\neq 0$.
In particular,
\begin{align}
{D}\left(\begin{array}{c}
e_1\\
e_2\\
e_3
\end{array}\right)=\left(\begin{array}{ccc}
0&0&0\\
0&-(\beta^2+\delta^2)&-(\beta^2+\delta^2)\\
0&\beta^2+\delta^2&\beta^2+\delta^2
\end{array}\right)\left(\begin{array}{c}
e_1\\
e_2\\
e_3 \notag \\
\end{array}\right).
\end{align}
\end{thm}
\begin{proof}
By (2.16) and (3.33), we get
\begin{align}
\left\{\begin{array}{l}
\alpha(\alpha^2+\beta^2+\beta\gamma+c)+\frac{\gamma^2\delta}{2}=0,\\
\beta(\alpha^2+\beta^2+c+\delta\alpha)+\frac{\gamma\delta^2}{2}=0,\\
\beta(3\alpha^2+\beta^2+\beta\gamma+\frac{\gamma^2}{2}+c)+(\delta-\alpha)(\alpha\beta+\frac{\gamma\delta}{2})=0,\\
\alpha(-\beta^2+\frac{\gamma^2}{2}-c)=0,\\
\beta(\beta^2-\beta\gamma-\frac{\gamma^2}{2}+c+\alpha\delta)=0,\\
\beta(\alpha^2+\beta^2+\alpha\delta+c)-\frac{\alpha\gamma\delta}{2}=0£¬\\
\gamma(-\beta^2+\frac{\gamma^2}{2}-\beta\gamma-c+\alpha\delta)=0,\\
\delta(\beta^2-\frac{\gamma^2}{2}+c)=0,\\
\delta(\alpha^2+\beta^2-\frac{\gamma^2}{2}+\beta\gamma+c)=0.\\
\end{array}\right.
\end{align}
\noindent Case 1) $\alpha=0$. By the first equation in (3.34) and $\alpha+\delta\neq 0$, we get $\delta\neq 0$ and $\gamma=0$. By (3.34), we get
$\beta^2+c=0$.\\
\noindent Case 2) $\alpha\neq 0$. By $\alpha\gamma=0$, then $\gamma=0$. By the eighth equation and the ninth equation in (3.34), we get $\delta=0$.
By the first equation in (3.34), we have $\alpha^2+\beta^2+c=0.$ By the third equation in (3.34), we have $\beta(2\alpha^2+\beta^2+c)=0$. So $\beta=0$
and $\alpha^2+c=0.$ By the fourth equation in (3.34), we have $\beta^2+c=0$, then $c=0$ and $\alpha=0$. This is a contradiction.
\end{proof}
By (2.12), (2.13) and (3.32), we get
$
{\widetilde{{\rm Wan}}}^0\left(\begin{array}{c}
e_1\\
e_2\\
e_3
\end{array}\right)=$
\begin{align}\left(\begin{array}{ccc}
-(\alpha^2-\frac{\beta\gamma}{2}+\beta^2)&-\alpha\beta&-(\alpha\beta+\frac{\gamma\delta}{4}-\frac{\gamma\alpha}{2})\\
-\alpha\beta&-(\alpha^2+2\beta^2+\delta^2+\frac{3}{2}\beta\gamma)&-(\frac{\alpha^2}{2}+\beta^2+\frac{3\beta\gamma}{4}+\delta^2)\\
\alpha\beta+\frac{\gamma\delta}{4}-\frac{\gamma\alpha}{2}&\frac{\alpha^2}{2}+\beta^2+\frac{3\beta\gamma}{4}+\delta^2
&\alpha^2+\beta\gamma+\frac{\gamma^2}{2}+\delta^2
\end{array}\right)\left(\begin{array}{c}
e_1\\
e_2\\
e_3
\end{array}\right).
\end{align}
\noindent If $(G_7,g,J)$ is the second kind algebraic Wanas soliton associated to the connection $\nabla^0$, then
$\widetilde{{\rm Wan}}^0=c{\rm Id}+D$, so
\begin{align}
\left\{\begin{array}{l}
De_1=-(\alpha^2-\frac{\beta\gamma}{2}+\beta^2+c)e_1-\alpha\beta e_2-(\alpha\beta+\frac{\gamma\delta}{4}-\frac{\gamma\alpha}{2})e_3,\\
De_2=-\alpha\beta e_1-(\alpha^2+2\beta^2+\delta^2+\frac{3}{2}\beta\gamma+c)e_2-(\frac{\alpha^2}{2}+\beta^2+\frac{3\beta\gamma}{4}+\delta^2)e_3,\\
De_3=(\alpha\beta+\frac{\gamma\delta}{4}-\frac{\gamma\alpha}{2})e_1+(\frac{\alpha^2}{2}+\beta^2+\frac{3\beta\gamma}{4}+\delta^2)e_2+
(\alpha^2+\beta\gamma+\frac{\gamma^2}{2}+\delta^2-c)e_3.\\
\end{array}\right.
\end{align}
\noindent By (2.16) and (3.36), we get
\begin{align}
\left\{\begin{array}{l}
\alpha(\frac{\alpha^2}{2}+\beta^2+\frac{7}{4}\beta\gamma+c)+(\beta+\gamma)(\frac{\gamma\delta}{4}-\frac{\gamma\alpha}{2})=0,\\
\beta(\beta^2-\frac{\beta\gamma}{2}+c)+\delta(\alpha\beta+\frac{\gamma\delta}{4}-\frac{\gamma\alpha}{2})=0,\\
\beta(2\alpha^2+\beta^2+\frac{\beta\gamma}{2}+\frac{\gamma^2}{2}+c)+(\delta-\alpha)(\alpha\beta+\frac{\gamma\delta}{4}-\frac{\gamma\alpha}{2})=0,\\
\alpha(\frac{\alpha^2}{2}-\beta^2+\frac{\gamma^2}{2}-\frac{\beta\gamma}{4}-c)-\frac{\beta\gamma\delta}{4}=0£¬\\
\beta(\alpha^2-\beta^2+\frac{\gamma^2}{2}+\frac{3}{2}\beta\gamma-\alpha\delta-c)=0,\\
\beta(\frac{\beta\gamma}{2}+\beta^2+\alpha\delta+c)+\alpha(\frac{\gamma\alpha}{2}-\frac{\gamma\delta}{4})=0,\\
\gamma(\frac{\alpha^2}{2}-\beta^2+\frac{\gamma^2}{2}-\frac{\delta^2}{4}+\frac{3}{4}\alpha\delta-\beta\gamma-c)=0,\\
\delta(\frac{\alpha^2}{2}-\beta^2+\frac{\gamma^2}{2}+\frac{\beta\gamma}{2}-c)+\frac{\alpha\beta\gamma}{2}=0,\\
\alpha(-\frac{\alpha^2}{2}-\beta^2-\frac{3}{4}\beta\gamma-c)+\alpha\beta\gamma+(\gamma+\beta)(\frac{\gamma\delta}{4}-\frac{\gamma\alpha}{2})=0.\\
\end{array}\right.
\end{align}
\noindent Solving (3.37) and similar to Theorem 3.10, Then
\vskip 0.5 true cm
\begin{thm}
$(G_7,g,J)$ is the second kind algebraic Wanas soliton associated to the connection $\nabla^0$ if and only if $\alpha=\gamma=0$, $\beta^2+c=0$, $\delta\neq 0$.
In particular,
\begin{align}
{D}\left(\begin{array}{c}
e_1\\
e_2\\
e_3
\end{array}\right)=\left(\begin{array}{ccc}
0&0&0\\
0&-(\beta^2+\delta^2)&-(\beta^2+\delta^2)\\
0&\beta^2+\delta^2&\beta^2+\delta^2
\end{array}\right)\left(\begin{array}{c}
e_1\\
e_2\\
e_3 \notag \\
\end{array}\right).
\end{align}
\end{thm}

\vskip 0.5 true cm

\section{Acknowledgements}

The author was supported in part by NSFC No.11771070.

\vskip 0.5 true cm


\bigskip
\bigskip

\noindent {\footnotesize {\it Y. Wang} \\
{School of Mathematics and Statistics, Northeast Normal University, Changchun 130024, China}\\
{Email: wangy581@nenu.edu.cn}

\end{document}